\documentclass[12pt]{article}
\usepackage{amsfonts}
\usepackage{latexsym}
\usepackage{amsmath}
\usepackage{amssymb}
\usepackage{amsthm}
\usepackage{fullpage}
\usepackage{graphicx}
\usepackage{enumerate}
\newtheorem{theorem}{Theorem}[section]
\newtheorem{problem}[theorem]{Problem}
\newtheorem{lemma}[theorem]{Lemma}
\newtheorem{fact}[theorem]{Fact}
\newtheorem{proposition}[theorem]{Proposition}
\newtheorem{corollary}[theorem]{Corollary}
\newtheorem{remark}[theorem]{Remark}

\theoremstyle{definition}
\newtheorem{definition}[theorem]{Definition}

\theoremstyle{remark}
\newtheorem*{note*}{Note}



\begin{document}

\small

\title{\bf On the equivalence between two problems of asymmetry on convex bodies}

\medskip

\author {Christos Saroglou}

\date{}

\maketitle

\begin{abstract}
\footnotesize The simplex was conjectured to be the extremal convex body for the two following ``problems of asymmetry'':\\
P1) What is the minimal possible value of the quantity $\max_{K'} |K'|/|K|$? Here, $K'$ ranges over all symmetric convex bodies contained in $K$.\\
P2) What is the maximal possible volume of the Blaschke-body of a convex body of volume 1?\\
Our main result states that (P1) and (P2) admit precisely the same solutions. This complements a result from [{\rm  K. B\"{o}r\"{o}czky, I. B\'{a}r\'{a}ny, E. Makai Jr. and J. Pach},
Maximal volume
enclosed by plates and proof of the chessboard conjecture],
stating that if the simplex solves (P1) then the simplex solves (P2) as well.
\end{abstract}

\section{Introduction}
\hspace*{1.5em}Let $K$ be a convex body in $\mathbb{R}^n$. The goal of this paper is to
study some properties of the extremal convex bodies for the following problem.
\begin{problem}\label{problem 1}Among all convex bodies $K$, find the one that minimizes the quantity
$$m(K)=\max_{K'} \frac{|K'|}{|K|} \ ,$$
where $K'$ ranges over all symmetric convex bodies contained in $K$.
\end{problem}
Here, $|\cdot|$ is the volume functional in $\mathbb{R}^n$.
However, the determination of these extremals is a difficult and long standing problem. Besicovitch \cite{Be}
proved that triangles the extremal bodies in the plane (see also \cite{CH-ST}).
The problem remains open in higher dimensions.
However, an asymptotically sharp estimate due to Stein \cite{St} is valid:
$m(K)>2^{-n}$.
Improvements of this result (with the same quantitative estimate) were established in \cite{CH-ST} and \cite{Mi-Pa}.

A remarkable result concerning this problem is due to F\'{a}ry and R\'{e}dei \cite{F2}.
They proved that there exists a unique symmetric convex body $K'\subseteq K$ of maximal volume,
called ``the symmetric kernel of K''. The center of $K'$ is often called ``the pseudo-center of $K$''
(see e.g. \cite{Mo-Zu}). We will denote this by $Ps(K)$.

Define the quantity
$$q(K):=\max_{x\in \mathbb{R}^n}|(K+x)\cap -K|= \max_{x\in \mathbb{R}^n}|(K+x)\cap -(K+x)|\ .$$
Then,
$$m(K)=\frac{q(K)}{|K|} \ .$$
Note that by the uniqueness result of F\'{a}ry and R\'{e}dei (mentioned previously), there exists a unique point $x_0\in\mathbb{R}^n$,
such that $|(K+x_0)\cap -K|=\max_{x\in \mathbb{R}^n}|(K+x)\cap -K|$. Since $(K+x_0)\cap -K=[(K+x_0/2)\cap-(K+x_0/2)]+x_0/2$, it follows that $x_0/2=Ps(-K)$.
Set $Q(K)=2Ps(-K)$. Clearly, the mapping $K\mapsto Q(K)$ is well defined and it is continuous with respect to the Hausdorff metric.
For the symmetric kernel $K'$ of $K$ we have
\begin{equation}\label{pseudocenter}
K'=(K+Q(K))\cap -K \ .
\end{equation}

As noted by Besicovitch \cite{Be}, $m(K)$ measures the asymmetry of $K$. We refer to the survey of Gr\"{u}nbaum \cite{Gru}
for an extensive discussion on the topic of measures of asymmetry
(for recent developments, see e.g. \cite{Me-Sc-We} \cite{Me-Sc-We-2} \cite{Gr1} \cite{Gr2} \cite{Gu}).
Let us consider another measure of asymmetry.
The Blaschke-body $\nabla K$ of $K$
is the unique origin-symmetric convex body whose surface area measure is given by:
$$S_{\nabla K}(\cdot )=\frac{1}{2}[S_{K}(\cdot )+S_{-K}(\cdot )] \ ,$$
where $S_K$ is the surface area measure of $K$ as defined on $S^{n-1}$ (see the next section).
The existence and uniqueness of $\nabla K$ are ensured by the Minkowski Existence Theorem,
stating that any measure on $S^{n-1}$, whose centroid is 0
and the affine hull of its support is full dimensional,
is the surface area measure of a unique (up to translation) convex body. It is true (see \cite{Kn-S}) that $|\nabla K|\geq |K|$,
with equality if and only if $K$ is symmetric. Thus, the quantity
$|\nabla K|/|K|$ is indeed a measure of asymmetry. The following problem arises naturally.
\begin{problem}\label{problem 2}Among all convex bodies $K$, find the one that maximizes the quantity $|\nabla K|/ |K|$.
\end{problem}
The study of the Blaschke body of $K$ (see \cite{Ga2}) is
related to Nakajima's problem \cite{Na}, asking whether a convex body of constant width and constant brightness has to be a ball.
See \cite{Ho}, \cite{Ho-Hu}, \cite{Hu} for newer results on this problem.

Introduce the quantities
$$m_n:=\inf\{m(K):K \textnormal{ is a convex body in } \mathbb{R}^n \}\ ,$$
$$M_n:=\sup\{|\nabla K|/|K|:K\textnormal{ is a convex body in }\mathbb{R}^n\} \ .$$
Clearly, the functionals $m(K)$ and $|\nabla K|/|K|$ are affine invariant (see e.g. \cite{Sch} for the second).
Since they are also continuous with respect to the Hausdorff distance,
the existence of convex bodies for which the quantities $m_n$ and $M_n$ are attained
(i.e. the existence of solutions to Problems \ref{problem 1} and \ref{problem 2}) follows easily by the Blaschke selection theorem.
It has repeatedly been conjectured (see e.g. \cite{F1} \cite{F2} \cite{BLYZ}) that
Problems \ref{problem 1} and \ref{problem 2} admit only one solution: the simplex. Problem \ref{problem 2} is open as well; in two dimensions it is confirmed \cite{BLYZ}
that the simplex is the only
solution. In addition, in \cite{BLYZ}, the following was established: If the simplex is a solution for Problem \ref{problem 1},
then it solves Problem
\ref{problem 2} as well. Moreover, $M_n\leq m_n^{\frac{-1}{n-1}}$. We are now ready to state our main results.
\begin{theorem}\label{mainthm1}
Let $K$ be a convex body in $\mathbb{R}^n$. $K$ is a solution for Problem \ref{problem 1} if and only if $K$ is a solution for Problem \ref{problem 2}.
\end{theorem}
\begin{theorem}\label{mainthm2}
If $K$ is a solution for Problem \ref{problem 1} or for Problem \ref{problem 2} and $|K\cap -K|=q(K)$, then
\begin{equation}\label{th2eq}
K\cap -K=m_n^{\frac{1}{n-1}}\nabla K \ .
\end{equation}
\end{theorem}
\begin{corollary}
$M_n=m_n^{\frac{-1}{n-1}}$.
\end{corollary}
\begin{proof} Take $K$ of volume 1 to be a solution of Problem \ref{problem 1}, such that $|K\cap-K|=q(K)$. Then, by Theorem \ref{mainthm1},
$|\nabla K|=M_n$ and the assertion follows immediately by taking volumes in (\ref{th2eq}).
\end{proof}
\begin{corollary} \cite{Be} (The planar case) The triangle is the only solution for Problem \ref{problem 1}.
\end{corollary}
\begin{proof}Let $K$ be a two-dimensional convex body. It is well known that in $\mathbb{R}^2$, $\nabla K=\frac{1}{2}(K-K)$, so by the Rogers-Shephard inequality \cite{RS},
$\frac{|\nabla K|}{|K|}$ is maximal if and only if $K$ is a simplex. This was remarked in \cite{BLYZ}. Thus, the triangle is the only solution for Problem \ref{problem 2}
in two dimensions and (by Theorem \ref{mainthm1}) it is the only solution for Problem \ref{problem 1} as well.
\end{proof}
\section{Background}
\hspace*{1.5em}We will need some basic results about convex bodies. We refer to \cite{Sch} for an extensive discussion, proofs and references concerning the facts that will
be mentioned in this section.

Let $K$ be a convex body in $\mathbb{R}^n$. The support function of $K$ at $x\in \mathbb{R}^n$ is defined as
$$h_K(x)=\max_{y\in K}\langle x,y\rangle \ ,$$
where $\langle \cdot,\cdot \rangle$ is the (usual) inner product in $\mathbb{R}^n$. Note that if $K$ contains 0 in its interior,
$F$ is a facet of $K$ and $u$ is the outer unit normal vector of $F$, then $h_K(u)$ is exactly the distance of $F$ from the origin.
It should be remarked that any convex and positively homogeneous function $h:\mathbb{R}^n\rightarrow \mathbb{R}$ is a support function of a unique convex body.

Let $\Omega$ be a Borel subset of the unit sphere $S^{n-1}$. The inverse spherical image
of $K$ at $\Omega$ is the set:
\[\tau(K,\Omega)=
\big\{x\in\textnormal{bd}K:\exists u\in\Omega,\textnormal{ such that }\langle x,u\rangle=h_K(u)\big\}\ .\]

The surface area measure of $K$ (viewed as a measure on $S^{n-1}$) is defined as
$$S_K(\Omega)={\cal{H}}^{n-1}\big(\tau(K,\Omega)\big)\ ,\ \Omega\textnormal{ Borel subset of } S^{n-1}\ . $$
Here, ${\cal{H}}^{n-1}(\cdot)$ stands for the $(n-1)$-dimensional Hausdorff measure.

A fact that will be used subsequently is that whenever a sequence of convex bodies converges,
in the sense of the Hausdorff distance, to a convex body $K$, then the corresponding sequence of the surface area measures converges
weakly to the surface area measure of $K$.

If $K$ is a polytope, the support of $S_K$ is exactly the set of the outer unit normal vectors of the facets of $K$.
Using this fact, one can easily see that
\begin{equation}\label{background-volume}
|K|=\frac{1}{n}\int_{S^{n-1}}h_K(x)dS_K(x) \ .
\end{equation}

Let $L$ be another convex body. The mixed volume $V(K,L,\dots ,L)$ of $K$ and $L$ is defined as the derivative of the quantity
$|tK+L|$, as $t \rightarrow 0^+$.
Here $A+B=\{x+y:x\in A, \ y\in B\}$ is the Minkowski sum of the sets $A$, $B$. It can be proven that
\begin{equation}\label{background-mixed-volume}
V_1(L,K)=V(L,\dots,L,K)=\frac{1}{n}\int_{S^{n-1}}h_K(x)dS_L(x) \ .
\end{equation}
Let us state two fundamental facts about mixed volumes:
The first is monotonicity; it is true that if $K\subseteq K'$, then $V_1(L,K)\leq V_1(L,K')$.
The second is the Minkowski inequality
\begin{equation}\label{background-Minkowski}
V_1(L,K)\geq |K|^{\frac{1}{n}}|L|^{\frac{n-1}{n}} \ .
\end{equation}
Equality here holds if and only if $K$ and $L$ are homothetic.

The projection body $\Pi K$ of $K$ is defined as the convex body whose support function along the direction $u\in S^{n-1}$
equals the $(n-1)$-dimensional volume of the orthogonal projection of $K$ in the same direction. It is true that
$$h_{\Pi K}(u)=\frac{1}{2}\int_{S^{n-1}}|\langle x,u\rangle |dS_K(x) \ .$$
This, together with Theorem \ref{mainthm2}, shows immediately the following:
\begin{corollary}\label{background-corollary}
If $|K\cap -K|=q(K)=m_n|K|$, then the projection bodies of $K\cap -K$ and $\nabla K$ are homothetic.
\end{corollary}
It is natural to ask the following:
\begin{problem}\label{problem homothety}For which (non-symmetric) convex bodies $K$, such that $Q(K)=0$, the projection bodies of $K$ and of $K\cap -K$ are homothetic?
\end{problem}
It is true that the simplex is such a convex body, however it is certainly not the only one as the examples of the regular polygons show.
\section{Proofs}

\hspace*{1 em}This section is devoted to the proof of Theorems \ref{mainthm1} and \ref{mainthm2}. The two theorems
will be proven simultaneously.
In what follows, every convergence of convex sets will be in the sense of the Hausdorff
distance.

Let $F:S^{n-1}\rightarrow (0,\infty)$ be a continuous function. The Wulff-shape
$W( F)$ associated with the function $F$ is the convex set defined by:
$$W(F)=\bigcap_{u\in S^{n-1}}G^-\big(u,F(u)\big)\ ,$$
where $G^-\big(u,k\big):=\{x\in\mathbb{R}^n:\langle x,u\rangle\leq k\}$, $u\in S^{n-1}$, $k\in\mathbb{R}$. Set, also, $G^+(u,k)$ to be
the complementary closed half-space of $G^-(u,k)$, i.e. $G^+(u,k)=\{x\in\mathbb{R}^n:\langle x,u\rangle\geq k\}$.

It follows by the definition that $W(F)$ is the maximal, with respect to inclusion, convex body, with support function less or
equal than $F$. It was shown by Aleksandrov that $$|W(F)|=\frac{1}{n}\int_{S^{n-1}}F(u)dS_{W(F)}(u)\ .$$

We will make use of the following lemma due to Aleksandrov \cite{A} (see also \cite{HLYZ}, \cite{BLYZ2} for further applications and
a generalization and
\cite{Sch} for
additional references).

\begin{lemma}\label{lemma-Alexandrov}\cite{A}
Let $F,G:S^{n-1}\rightarrow \mathbb{R}_+$ be continuous functions, where
$G$ is strictly positive. If $W(F)$ is a convex body (i.e. it is bounded), then
$W(F+tG)\rightarrow W(F)$, as $t\rightarrow 0^+$ and
$$\lim_{t\rightarrow 0^+}\frac{|W(F+tG)|-|W(F)|}{t}=\int_{S^{n-1}}G(u)dS_{W(F)}(u)\ .$$
\end{lemma}

Let $f:S^{n-1}\rightarrow\mathbb{R}_+$ be a continuous function. We will work with the following continuous deformation of $K$.
$$K_t(f):=W(h_K+tf) \ , \ t\geq 0.$$
It is clear that $K_t(f)$ contains $K$, for $t>0$ and that $K_0(f)=K$. Also, by Lemma \ref{lemma-Alexandrov},
$K_t(f)\rightarrow K$, as $t\rightarrow 0^+$. Following the idea of the proof of Lemma \ref{lemma-Alexandrov}, we will show:

\begin{lemma}\label{lemma-convergence}
Let $K$ be a convex body with $|K\cap -K|=q(K)$ and $f:S^{n-1}\rightarrow\mathbb{R}_+$ be a continuous even function. Then,
$$\lim_{t\rightarrow 0^+}\frac{|K_t(f)|-|K|}{t}=\int_{S^{n-1}}f(u) dS_K(u)$$
and $$\limsup_{t\rightarrow 0^+}\frac{\big|-K_t(f)\cap \big(K_t(f)+Q(K_t(f))\big)\setminus \big(-K\cap(K+Q(K_t(f)))\big)\big|}{t}
\leq \int_{S^{n-1}}f(u) dS_{K\cap -K}(u)\ .$$
\end{lemma}
\begin{proof}
The first assertion is an immediate consequence of Lemma \ref{lemma-Alexandrov}.
To prove the second assertion,
define the sets:
$$M_t:=-K_t(f)\cap \big(K_t(f)+Q(K_t(f))\big)\qquad \textnormal{and}
\qquad R_t:=-K\cap \big(K+Q(K_t(f))\big) \ .$$
Since $R_t\subseteq M_t$, we need to show that $\limsup_{t\rightarrow 0^+}(|M_t|-|R_t|)/t\leq \int_{S^{n-1}}f(u) dS_{K\cap -K}(u)$.
Note that $R_0=M_0=K\cap-K$ and that $M_t,R_t\rightarrow  K\cap-K$, with respect to the Hausdorff metric, as $t\rightarrow 0^+$.
For $u\in S^{n-1}$, $t\geq 0$, and since $f$ is even, we have:
\begin{eqnarray}
h_{M_t}(u)&\leq&\min\big\{h_{-K_t(f)}(u),h_{K_t(f)+Q(K_t(f))}(u)\big\}\nonumber\\
&\leq&
\min\big\{h_{K}(-u)+tf(-u),h_{K}(u)+\langle Q(K_t(f)),u  \rangle+tf(u)\big\}\nonumber\\
&=&\min\big\{h_{-K}(u),h_{K+Q(K_t(f))}(u)\big\}+tf(u)=:F_t(u)+tf(u)\ .
\label{eq1-lemma-convergence}
\end{eqnarray}
Since, clearly, $R_t$ is the maximal convex body whose support function is dominated by $F_t$, we have
$$|R_t|=\frac{1}{n}\int_{S^{n-1}}F_t(u)dS_{R_t}(u)\ ,\ t\geq 0\ .$$
Using (\ref{eq1-lemma-convergence}), (\ref{background-mixed-volume}) and the previous equation, we immediately obtain:
$$V_1(R_t,M_t)\leq \frac{1}{n}\int_{S^{n-1}}(F_t+tf)dS_{R_t}=|R_t|+t\frac{1}{n}\int_{S^{n-1}}fdS_{R_t}\ ,$$
thus
\begin{equation}
\limsup_{t\rightarrow 0^+}\frac{V_1(R_t,M_t)-|R_t|}{t}\leq\limsup_{t\rightarrow 0^+}\frac{1}{n}\int_{S^{n-1}}fdS_{R_t}=
\frac{1}{n}\int_{S^{n-1}}fdS_{K\cap-K}\ .\label{eq2-lemma-convergence}
\end{equation}

On the other hand, by the Minkowski inequality (\ref{background-Minkowski}), we get:
\begin{eqnarray}
& &\limsup_{t\rightarrow 0^+}\frac{V_1(R_t,M_t)-|R_t|}{t}\geq \limsup_{t\rightarrow 0^+}
\frac{|R_t|^{(n-1)/n}|M_t|^{1/n}-|R_t|^{(n-1)/n}|R_t|^{1/n}}{t}\nonumber\\
&=&|K\cap-K|^{(n-1)/n}\limsup_{t\rightarrow 0^+}\frac{|M_t|^{1/n}-|R_t|^{1/n}}{t}\ .\label{eq3-lemma-convergence}
%
\end{eqnarray}
Set $f(t)=|M_t|^{1/n}$ and $g(t)=|R_t|^{1/n}$. Then,
\begin{eqnarray*}
\limsup_{t\rightarrow 0^+}\frac{f^n(t)-g^n(t)}{t}&=&\limsup_{t\rightarrow 0^+}\frac{f(t)-g(t)}{t}
\big(f^{n-1}(t)+f^{n-2}(t)g(t)+\dots+g^{n-1}(t)\big)\\
&=&n|K\cap-K|^{(n-1)/n}\limsup_{t\rightarrow 0^+}\frac{f(t)-g(t)}{t}\ .
\end{eqnarray*}
This, together with (\ref{eq2-lemma-convergence}) and (\ref{eq3-lemma-convergence}) prove our claim.
\end{proof}

\begin{lemma}\label{lemma-properties}
Let $K$ be a convex body with $|K\cap-K|=q(K)=m_n|K|$. Set $A=\{v\in S^{n-1}:h_K(v)>h_{K\cap-K}(v)\}$. Then, $S_K(A)=0$.
\end{lemma}
\begin{proof}
Let $\Omega$ be a Borel subset of $S^{n-1}$. We will need the following notation:
$$\widetilde{K}(\Omega):=\bigcap_{u\in S^{n-1}\setminus \Omega}G^-\big(u,h_K(u)\big)\ .$$

\emph{Claim}: Let $\varepsilon>0$, $K$ be a convex body that contains 0 in its interior and $\Omega\subseteq S^{n-1}$ be a Borel set, with
$S_K(\Omega)>0$. Then, there exist a closed $\Omega'\subseteq \Omega$, $u\in \Omega'$, such that $S_K(\Omega')>0$
and
$$\widetilde{K}(\Omega')\setminus K\subseteq G^+\big(u,h_K(u)-\varepsilon\big)\ .$$
\emph{Proof of Claim}.
First choose a closed subset $\Omega_1$ of $\Omega$, with $S_K(\Omega_1)>0$.
Note that there exists $u\in \Omega_1$ and a sequence $\{O_m\}$ of open subsets of $S^{n-1}$,
such that $C_m:=cl\big(O_m\cap\Omega_1\big)\searrow \{u\}$ and $S_K(C_m)>0$.
If not, for every point $v$ in $\Omega_1$, there would exist an
open-in $\Omega_1$-set of $S_K$-measure 0, containing $v$. By compactness, $\Omega_1$ would be covered by a finite
collection of sets of $S_K$-measure 0, thus $S_K(\Omega_1)$ would be equal to 0, a contradiction.
It is true that $\{\tau(K,C_m)\}$ ($\tau(K,\cdot)$ was defined in Section 2) is a family of compact
sets, whose intersection equals the intersection of $K$ with its supporting hyperplane, whose outer unit normal vector is $u$.
Thus, for some large $m_0\in\mathbb{N}$, $\tau(K,C_{m_0})\subseteq G^+\big(u,h_K(u)-\varepsilon\big)$. Set $\Omega':=C_{m_0}$.
Let $x\in \textnormal{bd}K\cap \textnormal{int}G^-\big(u,h_K(u)-\varepsilon\big)$. Then, $K$ cannot be supported at $x$ by any
halfspace of the form $G^-\big(v,h_K(v)\big)$, $v\in \Omega'$, because otherwise $x$ would be contained in $\tau(K,\Omega')$.
Consequently, there exists $X\subseteq S^{n-1}\setminus \Omega'$, such that
$$\textnormal{int}G^-\big(u,h_K(u)-\varepsilon\big)\cap K=\textnormal{int}G^-\big(u,h_K(u)-\varepsilon\big)\cap
\bigcap_{v\in X}G^-\big(v,h_K(v)\big)\ ,$$
which shows that if a point $y$ is contained in $G^-\big(u,h_K(u)-\varepsilon\big)\setminus K$, then $y$ is contained in the interior of
$G^+\big(v,h_K(v)\big)$, for some $v\in S^{n-1}\setminus \Omega'$. This proves that $y\notin \widetilde{K}(\Omega')$, hence
the pair $(u,\Omega')$ satisfies the assertion
of our Claim. $\qed$

We are, now, ready to prove Lemma \ref{lemma-properties}.
If our assertion is wrong, there clearly exists a $0<\delta<\min_{v\in S^{n-1}}h_K(v)$, such that
$S_K(A_{\delta})>0$, where $A_{\delta}:=\{v\in S^{n-1}:h_K(v)-\delta>h_{K\cap-K}(v)\}$.
Then, by the previous Claim, we can find a closed set $\Omega\subseteq A_{\delta}$, such that $S_K(\Omega)>0$
and $\widetilde{K}(\Omega)\setminus K\subseteq G^+(u,h_K(u)-\delta/2)$, for some $u\in \Omega$. By the definition
of $A_{\delta}$, we have $K\cap-K\subseteq G^-\big(u,h_K(u)-\delta\big)$, so $-K\subseteq G^-\big(u,h_K(u)-\delta\big)$.
Thus, if $f:S^{n-1}\rightarrow \mathbb{R}_+$ is a continuous
function with $supp(f)\subseteq \Omega$ and $\int_{\Omega}fdS_K>0$, then
$K_t(f)\setminus K\subseteq \widetilde{K}(\Omega)\setminus K\subseteq
G^+(u,h_K(u)-\delta/2)$
and by the continuity of $Q(K_t(f))$, with respect to the Hausdorff distance, there exists $t_0>0$, such that
$$\big((K_{t_0}(f)\setminus K)+Q(K_{t_0}(f))\big)\cap -\big(K_{t_0}(f)\setminus K\big)=\emptyset
$$
and
$$\big(K_{t_0}(f)\setminus K\big)\cap -\big(K_{t_0}(f)+Q(K_{t_0}(f))\big)=
\big((K_{t_0}(f)\setminus K)+Q(K_t(f))\big)\cap -K_{t_0}(f)=\emptyset\ .$$ Therefore, we get
$$\big|\big(K_{t_0}(f)+Q(K_{t_0}(f))\big)\cap -K_{t_0}(f)\big|=
\big|\big(K+Q(K_{t_0}(f))\big)\cap -K\big|\leq m_n|K|<m_n|K_{t_0}(f)|\ .$$
This is a contradiction and our assertion is proved.
\end{proof}

\begin{lemma}\label{meta-main-lemma}
Let $K$ be a convex body. If $|K\cap -K|=q(K)=m_n|K|$, then $S_{K\cap -K}=m_nS_{\nabla K} $.
\end{lemma}
\begin{proof}
Assume that the assertion is not true. We distinguish two cases.

\emph{Case I}: $S_{K\cap-K}\ngeq m_nS_{\nabla K}$.\\
Then, there exists an even continuous function $f:S^{n-1}\rightarrow \mathbb{R}_+$, with $m_n\int_{S^{n-1}}fdS_{\nabla K}>\int_{S^{n-1}}fdS_{ K\cap -K}$. Then, by Lemma \ref{lemma-convergence},
we obtain:
\begin{eqnarray*}
& &\limsup_{t\rightarrow 0^+}\frac{\big|-K_t(f)\cap \big(K_t(f)+Q(K_t(f))\big)\setminus \big(-K\cap(K+Q(K_t(F)))\big)\big|}
{|K_t(f)|-|K|}\\
&\leq &\frac{\int_{S^{n-1}}f dS_{K\cap-K}}{\int_{S^{n-1}}f dS_{K}}=\frac{\int_{S^{n-1}}f dS_{K\cap-K}}{\int_{S^{n-1}}f dS_{\nabla K}}
<m_n\ .
\end{eqnarray*}
Set $B_t:=-K_t(f)\cap \big(K_t(f)+Q(K_t(f))\big)\setminus \big(-K\cap(K+Q(K_t(F)))\big)$. It follows that there exists $t>0$,
such that $|B_t|<m_n|K_t(f)\setminus K|$. Thus,
\begin{eqnarray*}
 \big|-K_t(f)\cap \big(K_t(f)+Q(K_t(f))\big)\big|&=&\big|\big(-K\cap(K+Q(K_t(F)))\big|+|B_t|\\
 &\leq& m_n|K|+|B_t|
 <m_n|K|+m_n|K_t(f)\setminus K|=m_n|K_t(f)|\ .
\end{eqnarray*}
This is a contradiction, so Case I cannot occur.

\emph{Case II}: $S_{K\cap-K}\geq m_nS_{\nabla K}$ and there exists $\Omega_1\subseteq S^{n-1}$, such that $S_{K\cap-K}(\Omega_1)
> m_nS_{\nabla K}(\Omega_1)$.\\
It follows by the assumption of Case II, Lemma \ref{lemma-properties} and (\ref{background-volume}), that
\begin{eqnarray*}
|K\cap-K|&=&\frac{1}{n}\int_{S^{n-1}}h_{K\cap-K}dS_{K\cap-K}\\
&>&\frac{m_n}{n}\int_{S^{n-1}}h_{K\cap-K}dS_{\nabla K}\\
&=&\frac{m_n}{n}\bigg(\frac{1}{2}\int_{S^{n-1}}h_{K\cap-K}dS_K+\frac{1}{2}\int_{S^{n-1}}h_{K\cap-K}dS_{-K}\bigg)\\
&=&\frac{m_n}{n}\bigg(\frac{1}{2}\int_{S^{n-1}}h_KdS_K+\frac{1}{2}\int_{S^{n-1}}h_{-K}dS_{-K}\bigg)=m_n|K|\ .
\end{eqnarray*}
This contradicts our assumption and our lemma is proved.
\end{proof}
The following lemma will make use of the Minkowski inequality in the same way as in \cite[Theorem 5$'$]{BLYZ}.
\begin{lemma}\label{minkowski-lemma}
Let $K$ be a convex body in $\mathbb{R}^n$, with $q(K)=|K\cap-K|$. Then,
\begin{equation}\label{minkowski}
|K \cap -K|^{1/n}|\nabla K|^{(n-1)/n}\leq |K| \ .
\end{equation}
In particular, $M_n\leq m_n^{-1/(n-1)}$.
\end{lemma}
\begin{proof}
The Minkowski inequality (\ref{background-Minkowski}) (together with (\ref{background-mixed-volume}) and the definition of the
Blaschke-body) imply
\begin{eqnarray*}
|K \cap -K|^{1/n}|\nabla K|^{(n-1)/n}
&\leq&V_1(\nabla K,K\cap -K )
=\frac{1}{2}V_1(K,K\cap -K)+\frac{1}{2}V_1(-K,K\cap -K)\leq |K|\ .
\end{eqnarray*}
\end{proof}
\textbf{ }\\
Proof of Theorems \ref{mainthm1} and \ref{mainthm2}:\\
Let $K$ be a convex body with $|K\cap-K|=q(K)$.
First assume that $K$ solves Problem \ref{problem 1}. Then, by Lemma \ref{meta-main-lemma},
$$S_{K\cap-K}=m_nS_{\nabla K}=S_{{m_n}^{\frac{1}{n-1}}\nabla K}$$
and by the uniqueness of the solution in the Minkowski problem for even measures (see e.g. \cite{Sch}), it follows that
(\ref{th2eq}) holds. Taking volumes in (\ref{th2eq}), we obtain: $m_n|K|=m_n^{n/(n-1)}|\nabla K|$ and thus
$m_n^{-1/(n-1)}\leq |\nabla K|/|K|\leq M_n$.
Since we already know the reverse inequality (by Lemma \ref{minkowski-lemma}), it follows that $m_n^{-1/(n-1)}= M_n$.
Therefore, $|\nabla K|/|K|=M_n$, so $K$ is a solution to Problem \ref{problem 2}. It remains to show that if $K$ solves Problem
\ref{problem 2}, then $K$ solves Problem \ref{problem 1} as well. In this case, (\ref{th2eq}) will hold by our previous discussion.
By Lemma \ref{minkowski-lemma}, we get:
$$M_n=|\nabla K|/|K|\leq (|K|/|K\cap-K|)^{1/(n-1)}\leq m_n^{-1/(n-1)}=M_n\ .$$
This proves our last assertion. $\qed$
\begin{remark}(i)
Using similar variational arguments as in Lemma \ref{meta-main-lemma}, one can prove the following:
Suppose that $K\cap -K=q(K)$ and $K$ is an extremal body for Problem \ref{problem 1} or equivalently Problem \ref{problem 2}. Then, $K\cap -K$ contains no extreme points of $K$ in its interior.
This shows for example that the extremal bodies cannot have smooth boundary.\\
(ii) Our method shows that Problems \ref{problem 1} and \ref{problem 2} remain equivalent even restricted in certain closed
subclasses of the
class of all convex bodies. Such examples are the class of all polytopes with at most $N$ facets, $N\in\mathbb{N}$ or the class
of all convex bodies whose surface area measures are supported in a prescribed subset of $S^{n-1}$.
\end{remark}

\section{The projection body of the Blaschke-body of the simplex}
\hspace*{1.5 em}Schneider (1982) \cite{SC2} asked for the maximizers of the affine invariant $$P(K):=\frac{|\Pi(K)|}{|K|^{n-1}} \ .$$
His original conjecture stated that the maximum among centrally symmetric convex bodies is attained if $K$ is the $n$-dimensional cube $C_n$; in this case, $P(C_n)=2^n$.
Counterexamples were discovered by Brannen \cite{Br2}, \cite{Br1}.
He conjectured that the simplex is the only maximizer in the general case and the centrally symmetric convex body of maximal volume
contained in the simplex in the centrally symmetric case (see also \cite{Sa} for the proof of some other conjectures of Brannen concerning Schneider's problem).
The latter is (as observed in \cite{BLYZ}) homothetic to the Blaschke-body of the simplex, $\nabla \Delta_n$.
Below, we mention some observations about the role of the Blaschke-body of the simplex (which is, in some sense, conjectured to be the extremal body for Problems \ref{problem 1} and \ref{problem 2}) in the study of Schneider's problem.
\begin{fact} Suppose that $\nabla \Delta_n$ is the only maximizer of $P(\cdot)$ in the symmetric case and that the simplex is the only solution for Problem \ref{problem 2}.
Then $\Delta_n$ is the only maximizer of $P(\cdot)$ in the general case.
\end{fact}
\begin{proof} It is true that
\begin{eqnarray*}
P(K)
=\frac{|\Pi \nabla K|}{|K|^{n-1}}
&=&\frac{|\Pi \nabla K|}{|\nabla K|^{n-1}}\cdot \frac{| \nabla K|^{n-1}}{|K|^{n-1}}\\
&\leq& \frac{|\Pi \nabla \Delta_n|}{|\nabla \Delta_n|^{n-1}}\cdot \frac{| \nabla \Delta_n|^{n-1}}{|\Delta_n|^{n-1}}\\
&=&P(\Delta_n) \ ,
\end{eqnarray*}
with equality if and only if $K$ is a simplex.  \end{proof}
Let us discuss another question concerning Schneider's problem. As mentioned earlier, there exists a convex body $K$ with $P(K)>P(C_n)$.
It is natural to ask however if Schneider's conjecture is in some sense ``almost'' correct. It is well known that $$P(K)<A^n \ ,$$
for all convex bodies $K$, where $A>0$ is some absolute constant. What appears to be unknown is the following
\begin{problem}\label{fproblem1}Is it true that
the ratio $$\bigg(\frac{\max_{K}P(K)}{P(C_n)}\bigg)^{\frac{1}{n}}$$
tends to 1, as $n$ tends to infinity? Here $K$ runs over all symmetric convex bodies.
\end{problem}
\begin{fact}\label{counterexamples}
Let $\Delta_n$ be an $n$-dimensional simplex. Then,
\begin{equation}\label{last}
\frac{P(\nabla\Delta_n)}{c\sqrt{n}P(C_n)}\rightarrow 1 \  , \ \textnormal{as }n\rightarrow \infty\ ,
\end{equation}
where $c>0$ is some absolute constant. Thus, if the conjecture of Brannen in the symmetric case of Schneider's problem is correct,
then the previous question has a strong affirmative answer.
\end{fact}
To see that (\ref{last}) is correct, use the asymptotic formula (proven in \cite{BLYZ} \cite{F2})
\begin{equation}\label{volume of blaschke-body}
 \frac{|\nabla\Delta_n|}{|\Delta_n|}\sim \sqrt{\frac{3}{2}}e\Big(\frac{e}{2}\Big)^n\ ,
\end{equation}
and also note that it is not too difficult to compute that $$P(\Delta_n)=\frac{n^n(n+1)}{n!} \ .$$
\begin{remark}
As (\ref{last}) shows, for large dimensions, $P(\nabla\Delta_n)>2^n=P(C_n)$ (thus, the Blaschke-body of the simplex
indeed provides a counterexample to the original conjecture of Schneider).
To prove that the same is true in any dimension, one can argue similarly as in \cite[Lemma 3.3]{Sa} to show that $P(\Pi K)=2^n$,
for any symmetric convex body $K$ with at most $2(n+1)$-facets.
Then, one can use Schneider's trick (see again \cite{Sch}) that
$P(K)\geq P(\Pi K)$, with equality if and only if the bodies $K$ and $\Pi \Pi K$ are homothetic.
Since $\nabla \Delta_n$ has $2(n+1)$-facets
and it is well known that it is not the projection body of any convex body, our claim follows.
\end{remark}
Finally, take the polar body $\Pi^{\ast}K$ of $\Pi K$ (i.e.
the unit ball of the dual of the normed space that has $\Pi K$ as its unit ball)
and the affine invariant $R(K):=|\Pi^{\ast}K||K|^{n-1}$. It has been conjectured that $C_n$
minimizes $R(K)$ among all centrally symmetric convex bodies
(see e.g. \cite{Lut}). The non-symmetric version of the previous conjectured inequality
is indeed true \cite{Z} (see also \cite{Sch} for related results). One may consider the following analogue of Problem \ref{fproblem1}.
\begin{problem}\label{fproblem2}Is it true that
the ratio $$\bigg(\frac{\min_{K}R(K)}{P(C_n)}\bigg)^{\frac{1}{n}}$$
tends to 1, as $n$ tends to infinity? Again, $K$ runs over all symmetric convex bodies.
\end{problem}
One might think of the Blaschke body of the simplex as a natural candidate for minimizing $R(K)$ in the class of symmetric convex bodies.
However, this is not true,
at least for large dimensions. Nevertheless, the values $R(\nabla K)$ and $R(C_n)$ are still asymptotically close.
Indeed, using (\ref{volume of blaschke-body}) and the exact value of $R(\Delta_n)$
(see again \cite{Sch}), one easily obtains:
\begin{fact}There exists a constant $C>1$, such that
$$\frac{R(\Delta_n)}{R(C_n)}\rightarrow C \ ,\textnormal{ as }n\rightarrow \infty \ . $$
\end{fact}
\textnormal{}\\
{\bf Acknowledgement.} I would like to thank the referee(s) for many helpful suggestions and improvements, especially for discovering
a serious logical gap in a previous version of this manuscript, in particular in the following statement: Every solution for Problem \ref{problem 1}
is a limit of solutions for Problem \ref{problem 1}, restricted in the class of polytopes with at most $N$ facets,
as $N\rightarrow \infty$.

\bigskip

\noindent \textsc{Ch.\ Saroglou}: Department of Mathematics,
Texas A$\&$M University, 77840 College Station, TX, USA.

\smallskip

\noindent \textit{E-mail:} \texttt{saroglou@math.tamu.edu \ \&\ christos.saroglou@gmail.com}

\end{document}